\documentclass[a4paper,12pt]{article}
\usepackage[english]{babel}
\usepackage{amsmath, amsthm}
\usepackage{amssymb}

\newcommand{\1}{1\!\!\,{\rm I}}
\renewcommand{\lg}{\langle}
\newcommand{\rg}{\rangle}

\newcommand{\cE}{{\mathcal E}}
\newcommand{\cT}{{\mathcal T}}

\newcommand{\mbR}{{\mathbb R}}

\newcommand{\wt}{\widetilde}
\newcommand{\vf}{\varphi}
\newcommand{\ve}{\varepsilon}
\newcommand{\ov}{\overline}
\theoremstyle{plain}
\newtheorem{thm}{Theorem}
\newtheorem{lem}{Lemma}

\theoremstyle{definition}
\newtheorem{defn}{Definition}

\newtheorem{expl}{Example}

\theoremstyle{remark}

\newcommand{\Var}{\mathop{\rm Var}}

\begin{document}\large
\renewcommand{\proofname}{Proof.}
\begin{center}
{\Large\bf
On regularization of formal Fourier--Wiener
transform of the self-intersection local time of planar
Gaussian process
}\\[1cm]
A.A.Dorogovtsev, O.L.Izyumtseva
\end{center}
\vskip20pt
Abstract: Fourier--Wiener transform of the formal
expression for multiple self-intersection local time is described in
terms of the integral, which is divergent on the diagonals. The method of
regularization we use in this work related to regularization of
functions with non-integrable singularities. The strong local nondeterminism property, which is more restrictive than the property of local nondeterminism introduced by S.Berman is
considered. Its geometrical meaning in the construction of the regularization is investigated. As the example the problem of regularization is solved for the compact perturbation of the planar Wiener process.

Key words: multiple self-intersection local
time, Fourier--Wiener transform, local nondeterminism.

AMS Mathematical Subject Classification. 60G15, 60H40.

The present paper is devoted to the multiple self-intersection local
time for planar Gaussian process. To define it we use
Fourier--Wiener transform. Fourier--Wiener transform of the formal
expression for multiple self-intersection local time is described in
terms of the integral of the ratio of two functions, where
denominator turns to zero on the diagonals. That is why this
integral must be regularized in some way. The method of
regularization we use in this work related to regularization of
functions with non-integrable singularities \cite{1} in the theory
of generalized functions. To present such a regularization in the
case of an arbitrary Gaussian process we introduce the property of
strong local nondeterminism which plays key role in the construction
of the regularization. Among a large number of works devoted to the
self-intersection local time for the random processes, let us recall
the papers related to our work. The problem of regularization of
self-intersection local time for planar Wiener process described in
\cite{2, 3}. In \cite{2}  Dynkin for
$f_\ve(x)=\frac{1}{2\pi\ve}e^{-\frac{\|x\|^2}{2\ve}}, \ve>0,
x\in\mbR^2$ considered the expression
$$
T^w_{\ve, k}=\int_{\Delta_k}\prod^{k-1}_{i=1}
f_\ve(w(s_{i+1})-w(s_i))d\vec{s}, \
\Delta_k=\{0\leq s_1\leq\ldots\leq s_k\leq1\}
$$
which ``blows up''  when $\ve\to0+.$ He proved that under the
right choice of the coefficients $B^l_k(\ve),$ the random variable
$$
\cT^w_k=L_p\mbox{-}\lim_{\ve\to0+}
\left[
T^w_{\ve, k}+\sum^{k-1}_{l=1}B^l_k(\ve)T^w_{\ve, l}
\right]
$$
is well defined.

In \cite{3} J.Rosen showed that there exists
$$
\int_{\Delta_k}\prod^{k-1}_{i=1}
(\delta_0(w(s_{i+1})-w(s_i)))-E\delta_0(w(s_{i+1})-w(s_i))d\vec{s}:=
$$
$$
=L_2\mbox{-}\lim_{\ve\to0+}
\left[
\int_{\Delta_k}\prod^{k-1}_{i=1}
(f_\ve(w(s_{i+1})-w(s_i))-Ef_\ve(w(s_{i+1})-w(s_i)))d\vec{s}
\right].
$$
The existence of multiple points of paths of Brownian motion in the plane,
Markov processes in a complete metric space and Gaussian processes is proved in \cite{4}--\cite{6}
correspondingly. The concept of local nondeterminism for Gaussian process
is considered in [7, 9]. Fourier--Wiener transform of Brownian functionals is
widely discussed in \cite{9}.

 The work consists of three parts. The necessity of
regularization  of formal expression for Fourier--Wiener transform
of the self-intersection local time for planar Gaussian process is
established in section 1.  As an example the problem of
regularization is considered for the case of planar Wiener process.

In section 2 we introduce the modification of the local
nondeterminism property which we call the strong local
nondeterminism. Here we consider geometrical meaning of this
property describing the joint behavior of the increments of the
process. The main example of Gaussian process with the strong local
nondeterminism is a compact perturbation of the Wiener process. In
Section 2 we present examples of such processes arising as a
solution to Sturm--Liuville problem with the white noise in the
right part.

In the section 3 we present the main result of the article about the
regularization of the Fourier--Wiener transform for the
self-intersection local time of the planar Gaussian process.

{\bf1. Formal expression of Fourier--Wiener transform for self-intersection local
time of  Gaussian process}

Let $\{x(t); t\in[0, 1]\}$   be a continuous in the square mean planar
Gaussian process with the mean zero. The main object of our investigation
is the following expression
\begin{equation}
\label{eq1.1}
T^x_k=
\int_{\Delta_k}\prod^{k-1}_{i=1}
\delta_0(x(s_{i+1})-x(s_i))d\vec{s},
\end{equation}
where $\delta_0$ is the delta-function at the point zero. The expression
\eqref{eq1.1}  is the formal definition of $k$-multiple self-intersection
local time for the process $x$ on the time interval $[0; 1].$

We will consider in \eqref{eq1.1} an action of $\delta_0$ on the
functionals from the white noise and use for its study the
well-developed tools from Gaussian analysis. Suppose that $H$ is a
real separable Hilbert space. The inner product in $H$ is denoted by
$(\cdot, \cdot).$

Let $g\in C([0; 1], H)$   be a such function that the linear span of its values is
dense in $H.$   Consider two independent Gaussian white noises in $H:$
$\xi_1$ and $\xi_2$  \cite{11}. Recall that Gaussian white noise $\xi$  in
$H$ is a family of jointly Gaussian random variables $\{(h, \xi); h\in H\}$
linearly depending on $h$  and such that
$$
E(h, \xi)=0, \ E(h, \xi)^2=\|h\|.
$$
Define the Gaussian process $x$ as follows
$$
x(t)=((g(t), \xi_1), (g(t), \xi_2)).
$$

To investigate \eqref{eq1.1}   consider its Fourier--Wiener transform. For
$h_1, h_2\in H$  let us denote by $\cE(h_1, h_2)$   the stochastic exponent
$$
\cE(h_1, h_2)=e^{\lg h_1, \xi_1\rg+ \lg h_2, \xi_2\rg-\frac{1}{2}
(\|h_1\|^2+\|h_2\|^2)}.
$$
Recall the following definition \cite{12}.
\begin{defn}
\label{defn1.1}
$\cT(\alpha)(h_1, h_2):=E\alpha\cE(h_1, h_2) $    is called Fourier--Wiener
transform of the random variable $\alpha.$
\end{defn}

Let us notice that a delta-function of Gaussian random variable as a
generalized Gaussian functional was considered by the following
authors \cite{9, 13}. Let us give meaning to the expression
$\prod^{k-1}_{i=1}\delta_0(x(s_{i+1})-x(s_i))$ by approximation of
delta-function with the following family of functions
$$
f_\ve(x)=\frac{1}{2\pi\ve}e^{-\frac{\|x\|^2}{2\ve}}, \ve>0, x\in\mbR^2.
$$

Consider approximating values
$\prod^{k-1}_{i=1}f_\ve(x(s_{i+1})-x(s_i)).$
It is not difficult to prove that there exists a limit
$$
\cT\left(
\prod^{k-1}_{i=1}\delta_0(x(s_{i+1})-x(s_i))\right)(h_1,h_2):=
$$
$$
=\lim_{\ve\to0+}
\cT\left(
\prod^{k-1}_{i=1}f_\ve(x(s_{i+1})-x(s_i))\right)(h_1,h_2)=
$$
\begin{equation}
\label{eq1.3}
= \frac
{e^
{-\frac{1}{2}(A^{-1}_{t_1\ldots t_k}(\vec{u}_1, \vec{u}_1)+
A^{-1}_{t_1\ldots t_k}(\vec{u}_2, \vec{u}_2)
)}}{\Gamma_{t_1\ldots t_k}},
\end{equation}
 where
 $$
 \Delta g(t_l)=g(t_{l+1})-g(t_{l}),l=1,\ldots ,k-1,
 $$
$$\vec{u}_i=((\Delta g(t_1), h_i), \ldots, (\Delta g(t_{k-1}), h_i)),
i=1,2,
$$
$$
A_{t_1\ldots t_k}=
(\Delta g(t_l), \Delta g(t_j))^{k-1}_{lj=1}.
$$
During the whole  article  we use the following notations.
$\Gamma_{t_1\ldots t_k}$ is a Gram determinant constructed on $\Delta g(t_1), \ldots, \Delta g(t_{k-1}).$ Also we suppose, that the following condition is fulfilled. For any $0\leq t_1<t_2<\ldots <t_k\leq 1 $
$$
\Gamma_{t_1\ldots t_k}\ne0.
$$
$P_{t_1\ldots t_k}$ is a projection on the linear span of
$
(\Delta g(t_1) \ldots, \Delta g(t_{k-1})).
$
It can be checked that the following lemma holds.

\begin{lem}
\label{lem1.1}
$$
A^{-1}_{t_1\ldots t_k}(\vec{u}_1, \vec{u}_1)=\|P_{t_1\ldots t_k}h_1\|^2, 
$$
 if $A^{-1}_{t_1\ldots t_k}$ exists.
\end{lem}
\begin{proof}

$$
A^{-1}_{t_1\ldots t_k}(\vec{u}_1, \vec{u}_1)=
\frac{1}{\Gamma_{t_1\ldots t_k}}
\sum^{k-1}_{ij=1}(-1)^{i+j}M_{ij}
(\Delta g(t_i),h_1)(\Delta g(t_j), h_1),
$$
where $M_{ij}$ is the minor of the matrix $A_{t_1\ldots t_k}$
corresponding to a line $i$  and a column $j.$   Let us define
$B_{t_1\ldots t_k}$   as follows
$$
B_{t_1\ldots t_k}h_1=
\frac{1}{\Gamma_{t_1\ldots t_k}}
\sum^{k-1}_{ij=1}(-1)^{i+j}M_{ij}
(\Delta g(t_i),h_1)\Delta g(t_j).
$$
It is not difficult to check that

1) For any $h_1\perp\Delta g(t_1), \ldots, \Delta g(t_{k-1}) $
$$
B_{t_1\ldots t_k}h_1=0,
$$

2) For any $i=\ov{1, k-1}$
$$
B_{t_1\ldots t_k}\Delta g(t_i)=\Delta g(t_i).
$$
Conditions 1), 2) imply that
$$
B_{t_1\ldots t_k}=P_{t_1\ldots t_k}.
$$

 To finish the proof it is enough to note that
$$
A^{-1}_{t_1\ldots t_k}(\vec{u}_1, \vec{u}_1)=
(B_{t_1\ldots t_k}h_1, h_1).
$$
Lemma is proved.
\end{proof}

It follows from lemma \ref{lem1.1}   that for
$\vec{t}=(t_1, \ldots, t_k)\in\Delta_k:$
\begin{equation}
\label{eq1.4}
\cT\left(
\prod^{k-1}_{i=1}\delta_0(x(s_{i+1})-x(s_i))\right)(h_1,h_2)=
 \frac
{e^
{-\frac{1}{2}(\|P_{t_1\ldots t_k}h_1\|^2+\|P_{t_1\ldots t_k}h_2\|^2
)
}}{\Gamma_{t_1\ldots t_k}}.
\end{equation}

Consider expression \eqref{eq1.4}  in the case of planar Wiener process.
Here we use $H=L_2([0, 1]).$   Then one can define a
Wiener process as
$w(t)=((\1_{[0, t]}, \xi_1), (\1_{[0, t]}, \xi_2)),$ where $\xi_1$
and $\xi_2$  are independent white noises in $L_2([0, 1]).$
Now \eqref{eq1.4}    has the following form
\begin{equation}
\label{eq1.5}
\cT\left(
\prod^{k-1}_{i=1}\delta_0(w(s_{i+1})-w(s_i))\right)(h_1,h_2)=
 \frac
{e^
{-\frac{1}{2}(\sum^{k-1}_{i=1}\|P_{t_it_{i+1}}h_1\|^2+
\sum^{k-1}_{i=1}\|P_{t_it_{i+1}}h_2\|^2
)
}}{
\prod^{k-1}_{i=1}(t_{i+1}-t_i)
}.
\end{equation}

The next statement describes the regularization of \eqref{eq1.5}.

\begin{thm}
\label{thm1.1}
The following integral is finite
$$
\int_{\Delta_k} \frac { \sum_{M\subset\{1,\ldots, k-1\}}
(-1)^{|M|}e^ {-\frac{1}{2}(\sum_{i\in M}\|P_{t_it_{i+1}}h_1\|^2+
\sum_{i\in M}\|P_{t_it_{i+1}}h_2\|^2 ) }} {
\prod^{k-1}_{i=1}(t_{i+1}-t_i) }d\vec{t}.
$$
\end{thm}

\begin{proof}
It is enough to check that the following integral exists
$$
$$
$$
\int_{\Delta_k} \Bigg|\frac { \sum_{M\subset\{1,\ldots, k-1\}}
(-1)^{|M|}e^ {-\frac{1}{2}(\sum_{i\in M}\|P_{t_it_{i+1}}h_1\|^2) }}
{ \prod^{k-1}_{i=1}(t_{i+1}-t_i) }\Bigg|d\vec{t}=
$$
$$
$$
$$
\int_{\Delta_k}
\Bigg|
\frac
{
\prod^{k-1}_{i=1}
(e^
{-
\|P_{t_it_{i+1}}h_1\|^2}-1
)
}
{
\prod^{k-1}_{i=1}(t_{i+1}-t_i)
}
\Bigg|
dt\leq
$$
$$
\leq
\int_{\Delta_k}
\frac
{
\prod^{k-1}_{i=1}
\|P_{t_it_{i+1}}h_1\|^2}
{
\prod^{k-1}_{i=1}(t_{i+1}-t_i)
}
d\vec{t}=
$$
\begin{equation}
\label{eq1.7}
=
\int_{\Delta_k}
\frac
{
\prod^{k-1}_{i=1}
\left(\int^{t_{i+1}}_{t_i}h_1(s)ds\right)^2}
{
\prod^{k-1}_{i=1}(t_{i+1}-t_i)^2
}
d\vec{t}.
\end{equation}

Let us prove that the integral \eqref{eq1.7}  converges. It is
sufficient to consider the case $h_1\geq0.$ Let us check that
$$
\int^1_{t_{k-1}}
\frac
{
\left(\int^{t_{k}}_{t_{k-1}}h_1(s)ds\right)^2}
{
(t_k-t_{k-1})^2
}dt_k\leq C\|h_1\|^2.
$$

It is not difficult to see
$$
\int^1_{t_{k-1}}
\frac
{
\left(\int^{t_{k}}_{t_{k-1}}h_1(s)ds\right)^2}
{
(t_k-t_{k-1})^2
}dt_k=
$$
$$
=
\int^1_{t_{k-1}}
\iint^{t_k}_{t_{k-1}}
h_1(s_1)h_1(s_2)ds_1ds_2
\frac{1}{t_k-t_{k-1}}dt_k=
$$
$$
=
\iint^{1}_{t_{k-1}}
h_1(s_1)h_1(s_2)
\int^1_{s_1\vee s_2}
\frac{1}{(t_k-t_{k-1})^2}dt_k=
dt_kds_1ds_2\leq
$$
$$
\leq
\iint^{t_k}_{t_{k-1}}
h_1(s_1)h_1(s_2)
\frac{1}{(
{s_1\vee s_2}
-t_{k-1})}
ds_1ds_2=
$$
$$
=2
\int^1_{t_{k-1}}h_1(s_1)
\int^1_{s_1}
\frac{h_1(s_2)}{s_2-t_{k-1}}ds_2ds_1.
$$

Consider in $L_2([t_{k-1}; 1])$ integral operator with the kernel
$$
k(s_1, s_2)=\frac{1}{s_2-t_{k-1}}\1_{\{s_2>s_1\}}.
$$
Let us check that $k$  defines bounded operator in $L_2([t_{k-1};1])$
using Shur test \cite{10}.   If there exist positive functions
$p, q: [t_{k-1}; 1]\to(0,+\infty)$  and $\alpha, \beta$   such that
$$
\int^1_{t_{k-1}}
k(s_1, s_2)q(s_2)ds_2\leq \alpha p(s_1),
$$
$$
\int^1_{t_{k-1}}
k(s_1, s_2)p(s_1)ds_1\leq \beta q(s_2),
$$
then $k$  corresponds to bounded operator with the norm less or equal to
$\alpha\beta.$  Put
$$
p(s_1)=\frac{1}{\sqrt{s_1-t_{k-1}}}, \
q(s_2)=\frac{1}{\sqrt{s_2-t_{k-1}}}.
$$
Then
$$
\int^1_{t_{k-1}}
k(s_1, s_2)q(s_2)ds_2=
$$
$$
=
\int^1_{s_1}
\frac{1}{(s_2-t_{k-1})^{3/2}}ds_2 \leq
2\frac{1}{\sqrt{s_1-t_{k-1}}},
$$
$$
\int^1_{t_{k-1}}
k(s_1, s_2)p(s_1)ds_1=
\int^{s_2}_{t_{k-1}}
\frac{1}{\sqrt{s_1-t_{k-1}}}ds_1
\frac{1}{s_2-t_{k-1}}=
\frac{2}{\sqrt{s_2-t_{k-1}}}.
$$

So we get the following estimation
$$
2\int^1_{t_{k-1}}h(s_1)\int^1_{s_1}
\frac{h(s_2)}{s_2-t_{k-1}}ds_2ds_1\leq8\|h\|^2.
$$
It implies that
$$
\int_{\Delta_k}
\frac
{
\prod^{k-1}_{i=1}
\left(
\int^{t_{i+1}}_{t_i}h_1(s)ds
\right)^2
}
{
(t_{i+1}-t_i)^2
}
d\vec{t}=
$$
$$
=
\int_{\Delta_{k-1}}
\frac
{
\prod^{k-2}_{i=1}
\left(
\int^{t_{i+1}}_{t_i}h_1(s)ds
\right)^2
}
{
(t_{i+1}-t_i)^2
}
\cdot
\int^1_{t_{k-1}}
\frac
{
\left(
\int^{t_{k+1}}_{t_k}h_1(s)ds
\right)^2
}
{
(t_{k}-t_{k-1})^2
}
dt_k
d\vec{t}\leq
$$
\begin{equation}
\label{eq1.8}
\leq 8\|h_1\|^2
\int_{\Delta_{k-1}}
\frac
{
\prod^{k-2}_{i=1}
\left(
\int^{t_{i+1}}_{t_i}h_1(s)ds
\right)^2
}
{
(t_{i+1}-t_i)^2
}
d\vec{t}.
\end{equation}

By using the same arguments it can be checked that the expression
\eqref{eq1.8}  less or equal then
$(8\|h_1\|^2)^{k-1}.$

The theorem is proved.
\end{proof}

The main aim  of the article is to construct the regularization of
expression \eqref{eq1.4}  for general Gaussian process $x.$

Let us describe the properties of the Gaussian process which are
necessary for the application of the method of regularization
considered in the theorem \ref{thm1.1}.

{\bf2. Strong local nondeterminism property}. In this section we
introduce the condition under which we are able to prove the
existence of regularization for Fourier--Wiener transform of
self-intersection local time of Gaussian process. This property is a
little bit stronger then  local nondeterminism introduced by
S.Berman \cite{7}. As before, for $g\in C([0; 1], H)$ we define Gaussian process $x(t)=(g(t),
\xi)$   with the help of white noise $\xi$ in the Hilbert space $H.$
 In this
section we also suppose, that 

for any $0\leq t_1<t_2<\ldots <t_k\leq 1 $
$$
\Gamma_{t_1\ldots t_k}\ne0.
$$
\begin{defn}
\label{defn2.1}
The process $x$ is strongly locally nondeterministic if for any fixed $k$ and an arbitrary $M\subset\{1, \ldots, k-1\}$
\begin{equation}
\label{eq1}
\Gamma_{t_1\ldots t_k}\sim G(\Delta g(t_i), i\notin M)\prod_{i\in M}
\|\Delta g(t_i)\|^2,
\end{equation}
when $\max_{i\in M}t_{i+1}-t_{i}\to0.$
\end{defn}
 It is evident that the property of local nondeterminism follows from
\eqref{eq1}. But the condition of definition \ref{defn2.1} is more
restrictive. For example, the next lemma shows that strong local
nondeterminism is sufficient for a weak convergence to zero of
projections on the small increments of the process.
\begin{lem}
\label{lem1'}
Suppose, that process $x$ is strongly locally nondeterministic. Then
$$
\forall \ h\in H:  \ P_{t_1t_2}h\to0, \ t_2-t_1\to0.
$$
\end{lem}
\begin{proof}
It is enough to consider $P_{t_1t_2}g(t)$    for fixed $t.$   Suppose,
that $t\leq t_1<t_2.$  Apply the condition \eqref{eq1}  to the points
$0<t<t_1<t_2$    or $0<t=t_1<t_2.$  Note, that \eqref{eq1}    means that
the orthogonal component of $\Delta g(t_i)$  to the linear span of
$\{\Delta g(t_j); 1\leq j\leq k-1, j\ne i\}$  asymptotically coinside
with  $\Delta g(t_i).$  It means in particular, that for arbitrary $j\ne i$
$$
\frac{(\Delta g(t_j), \Delta g(t_i))}
{\|\Delta g(t_j)\|\|\Delta g(t_i)\|}\to0, t_{i+1}-t_i\to0.
$$
In our case we have
$$
\frac{( g(t), \Delta g(t_1))}
{\|\Delta g(t_1)\|}\to0, t_{2}-t_1\to0.
$$
Now suppose, that $0\leq t_1<t_2\leq t.$   Then it follows from
\eqref{eq1}  that
$$
\frac{ (g(t)-g(t_2)+g(t_1), \Delta g(t_1))}
{\|\Delta g(t_1)\|}\to0, t_{2}-t_1\to0.
$$
Since $\|\Delta g(t_1)\|\to0, t_{2}-t_1\to0, $  then again
$$
\frac{ (g(t), \Delta g(t_1))}
{\|\Delta g(t_1)\|}\to0, t_{2}-t_1\to0.
$$
The last case $t_1\leq t\leq t_2$   can be considered in the same way.
Lemma is proved.
\end{proof}

Let us recall that $x$  is locally nondeterministic on some open
interval $J$ \cite{7}  if and only if
$$
\lim_{c\to0+}\inf_{t_m-t_1\leq c}
G
\left(
\frac{x(t_1)}{({\Var x(t_1)})^{1/2}}, \ldots,
\frac{x(t_m)-x(t_{m-1})}{(\Var(x(t_m)-x(t_{m-1})))^{1/2}}
\right)>0,
$$
for $m\geq2$ and an arbitrary points which are ordered according to
their indices: $t_1<\ldots<t_m$  in $J.$   The next example shows
that there exist locally nondeterministic processes for which the
statement of lemma 2 does not hold.
\begin{expl}
\label{expl1i}
$$
x(t)=w(t)+\sqrt{t}\xi, \ t\in[0, 1],
$$
where $w$  is one dimensional Wiener process, $\xi$  is a standart
Gaussian random variable. Suppose that $w$   and $\xi$  are independent.
Let us check that $x$  is locally nondeterministic. To prove that let us notice
that for $0<t_1<\ldots<t_m\leq 1$
$$
\Var(x(t_i)-x(t_{i-1}))=t_i-t_{i-1}+(\sqrt{t_i}-\sqrt{t_{i-1}})^2, i=\ov{2,m},
$$
$$
\left(
\frac{x(t_i)-x(t_{i-1})}{({\Var (x(t_i)-x(t_{i-1}})))^{1/2}},
\frac{x(t_i)-x(t_{i-1})}{(\Var(x(t_i)-x(t_{i-1})))^{1/2}}
\right)=1
$$
and for $k\ne l$
$$
 \left( \frac{x(t_k)-x(t_{k-1})}{({\Var
(x(t_k)-x(t_{k-1}})))^{1/2}},
\frac{x(t_l)-x(t_{l-1})}{(\Var(x(t_l)-x(t_{l-1})))^{1/2}} \right)=
$$
$$
=
\frac
{(\sqrt{t_l}-\sqrt{t_{l-1}})(\sqrt{t_l}-\sqrt{t_{l-1}})}
{
\sqrt{t_k-t_{k-1}+(\sqrt{t_k}-\sqrt{t_{k-1}})^2}
\sqrt{t_l-t_{l-1}+(\sqrt{t_l}-\sqrt{t_{l-1}})^2}}=
$$
$$
=
\frac
{
(\sqrt{t_k}-\sqrt{t_{k-1}}
)
(\sqrt{t_l}-\sqrt{t_{l-1}}
)
}
{
\sqrt{
(\sqrt{t_k}-\sqrt{t_{k-1}})2\sqrt{t_k}
}
\sqrt{(\sqrt{t_l}-\sqrt{t_{l-1}})2\sqrt{t_l}}
}=
$$
$$
=\frac{1}{2}
\left(1-\frac{\sqrt{t_{k-1}}}{\sqrt{t_k}}\right)^{1/2}
\left(1-\frac{\sqrt{t_{l-1}}}{\sqrt{t_l}}\right)^{1/2}\to0,
t_m-t_1\to0.
$$
It implies that
$$
\lim_{c\to0+}\inf_{t_m-t_1\leq c}
G
\left(
\frac{x(t_1)}{({\Var x(t_1)})^{1/2}}, \ldots,
\frac{x(t_m)-x(t_{m-1})}{(\Var(x(t_m)-x(t_{m-1})))^{1/2}}
\right)=1>0.
$$
To check that the projection related to the increment of the process
$x$  on small time interval does not tend to zero, consider
$$
g(t_1)=\sqrt{t_1}e+\1_{[0, t_1]}.
$$
Then for $h=e\oplus 0$ we get
$$
\|P_{t_1}h\|^2=\frac{(\sqrt{t_1})^2}{t_1+t_1}=\frac{1}{2}\not\to0,  t_1\to0.
$$
\end{expl}
This example shows, that the strong local nondeterminism property
is more restrictive then local nondeterminism.

 Main example of the
process with the strong local nondeterminism property in this
article is the process of the kind
\begin{equation}
\label{eq2.2}
x(t)=((I+S)g^0(t), \xi),
\end{equation}
where $I$ is identity operator and $S$  is compact operator in
$L_2([0; 1]),$ such that $\|S\|<1,$$ g^0(t)=\1_{[0; t]}.$
\begin{expl}
\label{expl2.1} Consider the process
$$
x(t)=w(t)+u(t),
$$
where $w$ is a Wiener process in $\mbR$ such that $w(t)=(\1_{[0; t]},\xi),$ $\xi$ is a white noise in $L_2([0, \frac{\pi}{2}])$ and $u$ is a solution of the
following Sturm--Liouville problem\cite{8}
\begin{equation}
\label{eq2.3'}
\begin{cases}
u''+u=\xi\\
u(0)=0\\
u(\frac{\pi}{2})=0,
\end{cases}
\end{equation}
 The solution of
\eqref{eq2.3'} is given by the formula
$$
u(t)=( g(t, \cdot), \xi),
$$
where $g$ is Green's function.

It is not difficult to check that
$$
g(t, s)=-\cos t\sin s\1_{\{s<t\}}-\sin t\cos s\1_{\{s>t\}}.
$$
Since we want $g$  to describe the law of action operator $S$  on
$g^0(t),$  then
\begin{equation}
\label{eq2.2}
(Sg^0t)(u)=S\1_{[0, t]}(u)=-\cos t\sin u\1_{\{u<t\}}-
\sin t\cos u\1_{\{u>t\}}.
\end{equation}
By using \eqref{eq2.2}  we get
\begin{equation}
\label{eq2.3}
(Sf)(s)=\int^{\frac{\pi}{2}}_0
f(u)[\1_{[0, s]}(u)\sin u\sin s-\1_{[s,\frac{\pi}{2} ]}(u)
\cos u\cos s]du.
\end{equation}
It follows from \eqref{eq2.3} that $S$ is compact operator.
\end{expl}
The following lemma describes one of the properties of the process
$x.$

\begin{lem}
\label{lem4'}
$x$  is strongly locally nondeterministic.
\end{lem}

\begin{proof}
To prove the lemma let us check that
$$
\lim_{\begin{subarray}{l}
{\max i\in M}\\
t_{i+1}-t_{i}\to0
\end{subarray}}
\frac{\Gamma_{t_1\ldots t_k}}
{G(\Delta g(t_i), i\notin M)\prod_{i\in M}\|\Delta g(t_i)\|^2}=1.
$$
For an arbitrary $q\in L_2([0;1]),$ denote by
$\wt{q}=\frac{q}{\|q\|}.$

Gram determinant properties imply that
$$
\frac{\Gamma_{t_1\ldots t_k}}
{G(\Delta g(t_i), i\notin M)\prod_{i\in M}\|\Delta g(t_i)\|^2}=
$$
$$
= \frac{G\Bigg( \wt{\Delta g}(t_1),\ldots, \wt{\Delta g}(t_{k-1})
\Bigg)} { G(\wt{\Delta g}(t_i), i\notin M)}.
$$
Check that for $k\in M, l=\ov{1, k-1}, k\ne l$
$$
\Bigg( \wt{\Delta g}(t_k), \wt{\Delta g}(t_{l}) \Bigg)\to0,
$$
when $\max_{i\in M}t_{i+1}-t_{i}\to0.$

Notice that for any $h\in L_2([0, 1])$  and $\ve>0$   there exists
$\delta>0$  such that for any $i\in M: \ t_{i+1}-t_{i}<\delta.$
\begin{equation}
\label{eq1vst2} \Bigg| (h, \wt{\Delta g}^0(t_{i-1}))  \Bigg|<\ve.
\end{equation}
It implies that $\|S \Delta \wt{g^0}(t_{i-1})\|\to0,$   when
$\max_{i\in M}t_{i+1}-t_{i}\to0$  since $S$  is a compact operator.

By using \eqref{eq1vst2} we get
$$
\frac { |((I+S)\Delta {g}^0(t_{k-1}), (I+S)\Delta {g}^0(t_{l-1}))| }
{ \| (I+S)\Delta {g}^0(t_{k-1}) \| \| (I+S)\Delta {g}^0(t_{l-1}) \|
}
$$
$$
= \frac{ (S\Delta\wt{g}^0(t_{k-1}),\Delta\wt{g}^0(t_{l-1}))+
(\Delta\wt{g}^0(t_{k-1}), S\Delta\wt{g}^0(t_{l-1}))+
}{\phantom{ppppppp} }
$$
$$
\frac{+(S\wt{\Delta g}^0(t_{k-1}), S\Delta\wt{g}^0(t_{l-1}))} {
\sqrt{ 1+2(S\Delta\wt{g}^0(t_{k-1}),\Delta\wt{g}^0(t_{k-1}))+
(S\Delta\wt{g}^0(t_{k-1}), S\Delta\wt{g}^0(t_{k-1})) } }
$$
\begin{equation}
\label{2vst2}
\frac{\phantom{ppppp} }
{\sqrt{
1+2(S\Delta\wt{g}^0(t_{l-1}),\Delta\wt{g}^0(t_{l-1}))+
(S\Delta\wt{g}^0(t_{l-1}), S\Delta\wt{g}^0(t_{l-1}))
}
}\to0
\end{equation}
 when  $\max_{i\in M}t_{i+1}-t_{i}\to0.$

Since a value of the determinant does not change under an even number of
transpositions of rows and columns we suppose that $M=\{1,\ldots, l\}.$
Then
$$
\frac {G \Bigg( \wt{\Delta g}(t_1),\ldots, \wt{\Delta g}(t_{k-1})
\Bigg) } { G\Bigg(\wt{\Delta g}(t_i), i\notin M\Bigg) } =
$$
$$
=\frac { G \Bigg(\wt{\Delta g}(t_i), i\notin M\Bigg) + F
\Bigg(\wt{\Delta g}(t_i),M_{1j}, m^n_{ij}, i,j=\ov{1,k-1}, n=\ov{1,
l-1}\Bigg) } { G\Bigg(\wt{\Delta g}(t_i), i\notin M\Bigg) },
$$
where
$$
F \Bigg(\wt{\Delta g}(t_i),M_{1j}, m^n_{ij}, i,j=\ov{1,k-1},
n=\ov{1, l-1}\Bigg)=
$$
$$
= \sum^{k-1}_{j=2}(-1)^{1+j} \Bigg( \wt{\Delta g}(t_1),\wt{\Delta g}(t_{j}) \Bigg)M_{1j}+
$$
$$
+ \sum^{k-1}_{j=3}(-1)^{1+j} \Bigg( \wt{\Delta g}(t_2),\wt{\Delta g}(t_{j}) \Bigg)m^1_{2j}+\ldots+
$$
$$
+ \sum^{k-1}_{j=l+1}(-1)^{l+j} \Bigg( \wt{\Delta g}(t_l),\wt{\Delta g}(t_{j}) \Bigg)m^{l-1}_{lj}.
$$
Here $M_{1j}$  the minor of the matrix $\Bigg(\Bigg( \Delta\wt{
g}(t_i), \wt{\Delta g}(t_{j})
 \Bigg)\Bigg)^{k-1}_{ij=1}.$

$m^k_{ij}$   the minor of the same matrix after a deleting of $k$  rows and
$k$  columns.

Since $\|S\|<1$   the $(I+S)$  has a continuous inverse operator.

This and compactness of $S$ imply that $\inf_{\vec{t}} G(\wt{\Delta g}(t_i), i\notin M)>0.$  Consequently,
$$
\frac {G(\wt{\Delta g}(t_i), i\notin M)+ F \Bigg(\wt{\Delta g}(t_i),M_{1j}, m^n_{ij}, i,j=\ov{1,k-1}, n=\ov{1, l-1}\Bigg) }
{ G(\wt{\Delta g}(t_i), i\notin M) } \to1,
$$
when $\max_{i\in M}: t_{i+1}-t_{i}\to0. $

Lemma is proved.
\end{proof}

The strong local nondeterminism property can be reformulated in
terms of conditional variance.
\begin{defn}
\label{defn1'} Gaussian process $x$   has strong local nondeterminism
property if and only if $t_1<t_2<\ldots <t_k$
$$
\frac{\Var(\Delta x(t_i)/\Delta x(t_j), 1\leq j\leq k-1, j\ne i)}
{\Var(\Delta x(t_i))}\to1, t_{i+1}-t_i\to0.
$$

The strong local nondeterminism property can be used to describe asymptotic
behavior of $\Gamma_{t_1\ldots t_k}$ when some of differences
$t_{i+1}-t_i$ converge to zero. Note that this convergence holds for every $i=1,\ldots,k-1.$  In contrast to Berman  definition\cite{7}, where $i=k-1.$
\end{defn}

{\bf3. Regularization for Fourier--Wiener transform of
self-intersection local time}.

As it was shown in the previous section the
Fourier--Wiener transform of the formal expression for the self-intersection local
time contains the function $\Gamma^{-1}_{t_1\ldots t_k}$  which has
singularities along the diagonals. Here we present the way of
regularization of the integral with
 $\Gamma^{-1}_{t_1\ldots t_k}$ for the processes which are the compact
perturbations of the Wiener process. Let us suppose in this section,
that the Gaussian process $x$  has the following form
$$
x(t)=((g(t), \xi_1), (g(t), \xi_2))
$$
with the independent Gaussian white noises $\xi_1, \xi_2$   in
$L_2([0; 1])$   and
$$
g(t)=g^0(t)+Sg^0(t),
$$
where $g^0(t)=\1_{[0; t]},$ $S$ is a compact operator in $L_2([0; 1])$  with
$\|S\|<1.$  For $0\leq t_1<\ldots<t_k\leq1$  denote by
$\wt{\Delta g}(t_1),\ldots,\wt{\Delta g}(t_{k-1})$
the orthonormal
system which is obtained from
${\Delta g}(t_1),\ldots,{\Delta g}(t_{k-1})$
via the orthogonalization procedure. Since the elements
${\Delta g}(t_1),\ldots,{\Delta g}(t_{k-1})$
are linearly independent (see section 2)  all the elements
$\wt{\Delta g}(t_1),\ldots,\wt{\Delta g}(t_{k-1})$
are non-zero. For
$M\subset\{1, \ldots, k-1\}$
denote by $P_M$  the projection on
$\wt{\Delta g}(t_i), i\in M.$

\begin{thm}
\label{thm3.1'} The following integral converges for arbitrary
$h\in L_2([0; 1])$
$$
\int_{\Delta_k}\Gamma^{-1}_{t_1\ldots t_k}
(\sum_{M\subset\{1,\ldots, k-1\}}
(-1)^{|M|}e^{-\frac{1}{2}\|P_Mh\|^2}
)d\vec{t}.
$$
\end{thm}
\begin{proof}
It is enough to check the convergence of the integral
$$
\int_{\Delta_k}|\prod^{k-1}_{j=1}(t_{j+1}-t_j)^{-1}
(\sum_{M\subset\{1,\ldots, k-1\}}
(-1)^{|M|}e^{-\frac{1}{2}\|P_Mh\|^2}
)|d\vec{t}=
$$
$$
=
\int_{\Delta_k}\prod^{k-1}_{j=1}
\frac{1-e^{-\frac{1}{2}(h, \wt{\Delta g}(t_j))^2}}{t_{j+1}-t_j}d\vec{t}\leq
$$
$$
\leq
\int_{\Delta_k}
\frac{1}{2^{k-1}}
\prod^{k-1}_{j=1}
\frac{(h, \wt{\Delta g}(t_j))^2}{t_{j+1}-t_j}d\vec{t}.
$$
Let us consider
$$
\int^1_{t_{k-1}}
\frac{(h, \wt{\Delta g}(t_{k-1}))^2}{t_{k}-t_{k-1}}dt_k.
$$
Denote by $f(t_k)$  the difference
$$
f(t_k)=\Delta g(t_{k-1})-P_{t_1\ldots t_{k-1}}\Delta g(t_{k-1}).
$$
As it was proved before the process $x$ is strongly locally nondeterministic.
Hence, uniformly with respect to $t_1,\ldots, t_{k-1}$  the
following relations hold
$$
\|
P_{t_1\ldots t_{k-1}}
\frac{\Delta g(t_{k-1})}{\sqrt{t_k-t_{k-1}}}
\|\to0, \ t_k\to t_{k-1},
$$
$$
\frac{\|f(t_k)\|^2}{t_k-t_{k-1}}\to1, \ t_k\to t_{k-1}.
$$
Consequently, it is enough to consider integral
$$
\int^1_{t_{k-1}}
\frac{(h, {\Delta g}(t_{k-1}))^2}{(t_{k}-t_{k-1})^2}dt_k=
$$
$$
=
\int^1_{t_{k-1}}
\frac{(h+S^*h, {\Delta g^0}(t_{k-1}))^2}{(t_{k}-t_{k-1})^2}dt_k.
$$
It can be shown that the last integral can be estimated above by  $C\|h\|^2$  for some absolute
constant $C$  as it was done in Section 1. Consequently, the initial
integral absolutely converges. The theorem is proved.
\end{proof}

As a corollary one can obtain the regularization for the formal expression
of the self-intersection local  time for the process $x. $  Define for
$0\leq t_1<\ldots<t_k\leq1$  the random vectors
$\ov{\Delta x}(t_1),\ldots, \ov{\Delta x}(t_{k-1})$
as follows
$$
\ov{\Delta x}(t_{1})={\Delta x}(t_{1}),
$$
$$
\ov{\Delta x}(t_{j})={\Delta x}(t_{j})-E({\Delta x}(t_{j})/
{\Delta x}(t_{1}), \ldots, {\Delta x}(t_{j-1})), \ j=2, \ldots, k-1.
$$
The following statement holds.
\begin{thm}
\label{thm3.2'}
The following integral from the generalized Gaussian functional has a well-defined Fourier--Wiener
transform
$$
\int_{\Delta_k} \sum_{M\subset\{1,\ldots, k-1\}} (-1)^{|M|+(k-1)}
\prod_{j\in M} \delta_0(\ov{\Delta x}(t_j)) \prod^{k-1}_{j=1}
\frac{1}{t_{j+1}-t_j} (E\prod_{j\in M}\delta_0(\ov{\Delta
x}(t_j)))^{-1}
 d\vec{t}.
$$
\end{thm}

The proof of this theorem is a straightforward application of Theorem
\ref{thm3.1'}.

{\it Remark}. Note, that for Wiener process $\ov{\Delta x}(t_j)$
coincide with $\Delta x(t_j), j=1, \ldots, k-1$   and we obtain a regularization
described in the Section 1.

\end{document}